\documentclass{amsart}

\usepackage[utf8x]{inputenc}
\usepackage[english]{babel}
\usepackage{amsmath,amsthm,amssymb, latexsym,geometry,fancyhdr,stackrel,enumerate,bbm}
\usepackage{eufrak}
\usepackage[all]{xy}
\usepackage{srcltx}
\usepackage[left,modulo, pagewise]{lineno}
\usepackage{caption}
\usepackage{hyperref}
\usepackage{todonotes}
\usepackage{tikzsymbols}
\usepackage{pgf}
\usepackage{tikz-cd}
\usetikzlibrary{cd}
\usetikzlibrary{arrows}
\usetikzlibrary{topaths}
\usetikzlibrary{decorations.pathmorphing}

\date{\today}
\swapnumbers
\theoremstyle{plain}

\newtheorem*{Thm}{Theorem}
\newtheorem*{Lem}{Lemma}

\newtheorem*{Prop}{Proposition}

\theoremstyle{definition}

\newtheorem*{Rmk}{Remark}
\newtheorem*{Ex}{Example}

\newcommand{\K}{\Bbbk}

\renewcommand{\a}{\alpha}
\renewcommand{\b}{\beta}
\renewcommand{\c}{\gamma}

\newcommand{\f}{\varphi}
\newcommand{\tg}{\mathfrak{t}}
\newcommand{\so}{\mathfrak{s}}

\newcommand*{\Cdot}{\raisebox{-0.65ex}{\scalebox{2}{$\cdot$}}}
\newcommand{\vtx}[1]{\underset{#1}{\Cdot}}
\newcommand{\Vtx}[1]{\overset{#1}{\Cdot}}

\newcommand{\Hom}[3]{{\rm Hom}_{#1}(#2,\,#3)}
\newcommand{\Ext}[4]{{\rm Ext}_{#1}^{#2}(#3,\,#4)}

\newcommand{\HH}[2]{{\rm HH}^{#1}(#2)}

\newcommand{\AP}[2]{\Gamma^{#1}_{#2}}
\renewcommand{\mod}{{\rm mod}-}
\newcommand{\E}[1]{\mathbf{E}(#1)}
\newcommand{\Cy}[1]{\mathbf{C}(#1)}

\newcommand{\sm}[1]{\begin{smallmatrix}#1\end{smallmatrix}}

\newcommand{\csr}[2]{_{#1}\hspace{-1mm}\circ_{#2}}

\numberwithin{equation}{subsection}

\title[Gerstenhaber bracket  and cycles.]{The Gerstenhaber bracket and cycles in the module category of a monomial quadratic algebra}

\author[J.C. Bustamante]{Juan Carlos Bustamante}
\address[J.C. Bustamante]{Mathematics Department. Champlain College. 2580 College St, Sherbrooke, Qu\'bec J1M 2K3}
\email{jcbustamante@crc-lennox.qc.ca}

\author[P. Le Meur]{Patrick Le Meur}

\address[Patrick Le Meur]{Universit\'e Paris Diderot, Sorbonne
  Universit\'e, CNRS, Institut de Math\'ematiques de Jussieu-Paris
  Rive Gauche, IMJ-PRG, F-75013, Paris, France}
\email{patrick.le-meur@imj-prg.fr} \subjclass[2000]{Primary 16E40,
  16G60} \thanks{The second named author acknowledges financial
  support from CRM (UMI CNRS 3457), from CRSNG and from Universit\'e de
  Sherbrooke.}

\keywords{Hochschild cohomology; Gerstenhaber bracket}
\begin{document}

\begin{abstract}
We establish a link between the Gerstenhaber bracket in the Hochshild
cohomology and the behaviour of cycles in the module category of a
monomial quadratic algebra $A$.
\end{abstract}

\maketitle
\section*{Introduction}

Given a finite-dimensional algebra $A$ over a field $\K$, its
Hochschild cohomology groups, denoted by $\HH{i}{A}$, are
$\HH{i}{A} = \Ext{A^e}{i}{A}{A}$ where $A^e = A \otimes_\K A^{op}$ is
the enveloping algebra of $A$. The natural equivalence between
$A-A$-bimodules and right $A^e$-modules allows to see the groups
$\HH{i}{A}$ as extension spaces of the $A-A$-bimodule by itself. The
groups corresponding to the lower degrees admit interpretations in
terms of the center, the outer derivations, as well as infinitesimal
deformations of the algebra $A$.

Besides this, the sum $\HH{*}{A} = \oplus_{i\geqslant 0} \HH{i}{A}$ is
endowed with a very rich structure, that of a Gerstenhaber
algebra. That is, it is a graded commutative ring (for the
cup-product) which is equipped with a Lie bracket(of degree $-1$) and
the two structures are compatible, the bracket inducing graded
derivations with respect to the cup-product. The two operations are
defined using the standard Hochschild complex of $A$, see the
Gerstenhaber's paper \cite{G63}.

While the cup-product coincides with the Yoneda splicing of
extensions, the Gerstenhaber bracket is harder to interpret. Schwede
provided in
\cite{S98}\ 
an interpretation based on homotopy classes in extension
categories. Besides this rather abstract description, some concrete
computations and results have been established in the last years in
particular cases : string algebras \cite{Bus06-cohomo, RR18},
radical-square zero algebras \cite{SF08}, or some group algebras
\cite{NW16-MR3498646, SF12}. The main difficulty to study the
Gerstenhaber bracket comes from the fact that the usual bar resolution
is normally too big to carry out actual concrete
computations. However, recent developments show that the bracket can
be treated using alternate approaches or other resolutions
\cite{NW16-MR3498646, V19-MR3974969, SA17-MR3623179,
  NVW-Arxiv1805.03167}.

Our work in the present paper, however, points in a different
direction, perhaps more "representation - theoretic". Indeed, we are
mainly interested in the information about $\mod{A}$ that can be
recovered from the Hochschild cohomology of $A$. There are number of
results concerning the vanishing of the groups corresponding to the
lower degrees (see \cite{Skow92}), but as far as we know, there are no
known results concerning the additional structure of $\HH{*}{A}$.

Our idea originates in section 5.4 of \cite{Hap89}, where it is
established that for representation-finite algebras the groups
$\HH{i}{A}$ for $i\geqslant 2$ detect cycles in the module category
$\mod{A}$. Let us now explain the idea while we describe the
structure and contents of the paper.

We focus on monomial quadratic algebras, a class for which a very
convenient minimal projective resolution of $_AA_A$ (due to Bardzell
\cite{B97}) and an explicit comparison morphism \cite{Bus06-cohomo,
  RR-18-comp} with the standard bar resolution are known. This makes
possible to carry the operations defined using the standard Hochschild
complex to the Bardzell complex. The required details are given in
section 1 where in addition we recall the notions concerning algebras
and modules that will be needed in the sequel.

The minimal resolution is parametrized by the extension groups between
the simple $A$-modules. These extensions are known to be, after
Bardzell \cite{B97}, completely described by the so-called
"\emph{associated sequences of paths}", that we call by
$\AP{}{}$-paths. In section 2 we introduce combinatorial and
representation theoretic material for describing the defining
composition operations "$\circ_i$" of the Gerstenhaber bracket. This
includes the concepts of $\Gamma$-bypasses and their composition as
well specific cycles in $\mod{A}$ - called admissible cycles -
associated to $\Gamma$-bypasses together with a certain composition of
these cycles. On one hand, the extension spaces between simple modules
have bases parametrised by $\Gamma$-paths and the cup product of these
extensions can be translated in terms of compositions of specific
$\Gamma$-bypasses. On the other hand, $\Gamma$-bypasses parametrize
cochains on Bardzell's resolution with values in $A$; we prove that
the "$\circ_i$" composition of these cochains is encoded by the
composition of the associated admissible cycles.

The results of this paper were presented at the {\sc XXII}
\emph{Coloquio Latinoamericano de Álgebra}, that took place in Quito,
Ecuador in 2017.

\section{Preliminaries}
\subsection{Algebras and modules}
While we briefly recall some particular concepts concerning bound
quiver algebras, we refer the reader to \cite{ASS06}, for instance,
for unexplained notions.

Let $Q=(Q_0,Q_1,\mathfrak{s},\mathfrak{t})$ be a finite quiver and $\Bbbk$ a commutative
field. We consider algebras of the form $A = \Bbbk Q / I$ with $I$ an
admissible ideal of the path algebra $\Bbbk Q$. This includes, but
does not restrict to, all basic connected algebras over algebraically
closed fields. The composition of two arrows $\a : i \to j$ and
$\b : j \to l$ is denoted by $\a \b$. A two-sided ideal $I$ of the
path algebra $\Bbbk Q$ is said to be \emph{monomial} if it can be
generated by a set of paths. If there is a set of generators formed by
linear combinations of paths of length 2, then $I$ is said to be
\emph{quadratic}.  We will mainly focus on monomial and quadratic
algebras, and by abuse of notation will identify a path $w$ in $\K Q$ with its class $w + I$ in $A$.

Given an arrow $\a:x\to y$ we denote by $\a^ {-1} : y\to x$ its formal
inverse, and agree that $\left(\a^ {-1}\right)^ {-1} = \a$. A
\emph{walk} in $Q$ is a composition $w = \a_1 \cdots \a_n$ of arrows
or formal inverses of arrows (so the target of $\a_i$ is the source of
$\a_{i+1}$ for $i\in\{1,\ldots, n-1\}$). A walk $w$ as above is
\emph{reduced} if $\a_{i-1} \ne \a_i^{-1}$ for $i\in\{2,\ldots, n\}$.
Let $A = \Bbbk Q/I$ with $I$ monomial. A reduced walk $w$ in $Q$ is a
\emph{string} in $(Q,I)$ if neither $w$ nor $w^{-1}$ contain a subpath
that is a path in $I$. A string is \emph{direct} if it is composed
entirely by arrows of $Q$ and \emph{inverse} if it is composed only by
formal inverses of arrows. The trivial walks are both inverse and
direct.

To a string $w = \a_1 \cdots \a_n$ is associated the so-called
\emph{string module} $M(w)$. If $w =e_x$ is the trivial walk at some
vertex $x$, then $M(w)$ is simply the corresponding simple module
$S_x$.  Let $w = \a_1 \ldots \a_n$ be a string with $n\geqslant 1$.
For $i\in{1,\ldots, n+1}$, define $K_i = \Bbbk$ and define $\f_{\a_i}$
as the identity map $K_i \to K_{i+1}$ if $\a_i$ is an arrow, and the
identity map $K_{i+1} \to K_i$ in case $\a_i^ {-1}$ is an arrow of
$Q$.

The $A$-module $M(w)$ (equivalently the representation of $(Q,I)$) is
defined as follows:
\begin{itemize}
\item For a vertex $a\in Q_0$, $M(w)_a$ is the sum of all the vector
  spaces $K_i$ such that the source of $\a_i$ is $a$, together with
  $K_{n+1}$ in case the target of $\a_n$ is $a$. In particular, if $w$
  does not pass trough $a$, then $M(w)_a = 0$.

 \item Given $\c\in Q_1$ and $\epsilon\in \{\pm 1\}$, if $\c^\epsilon$
   appears in $w$, then $M(w)_\c$ is the direct sum of the maps
   $\f_{\a_i}$ such that $\a_i = \c^\epsilon$. If $\c^\epsilon$ does not
   appear in $w$, them $M(w)_\c = 0$.
\end{itemize}

Given a string $u= \a_1 \cdots \a_n$ and a substring
$v = \a_i \cdots \a_j$ (for some $i,j$ with
$1\leqslant i < j \leqslant n$) we also have that :

there is a monomorphism

\[ M(v) \to M(u) \text{ if and only if } \left\{ \begin{array}{lcl}
                                                   \a_{j+1} \text{ is inverse} & \text{ or } & j=n\\
                                                   \text{and} & & \\
                                                   \a_{i-1} \text{ is
                                                   direct} & \text{
                                                             or} &i=1.
\end{array}
 \right.\]

And, dually, there is an epimorphism

 \[ M(u) \to M(v) \text{ if and only if } \left\{ \begin{array}{lcl}
                                                    \a_{j+1} \text{ is direct}  & \text{ or } & j=n\\
                                                    \text{and}
                                                                                &
                                                                                              & \\
                                                    \a_{i-1} \text{ is
                                                    inverse} & \text{
                                                               or }
                                                                                              &i=1.
 \end{array}
  \right.\]

  In particular, if $\a : i \to j$ is an arrow, we have a monomorphism
  $S_j \to M(\a)$ and an epimorphism $M(\a) \to S_i$.

  \subsection{Gerstenhaber bracket in Hochschild
    cohomology}\label{sec:reduced-bar}
  In \cite{G63} Gerstenhaber defined a Lie bracket $[ -, - ]$ on
  $\HH{*}{A}$. The original definition was given in terms of the
  complex obtained upon applying $\Hom{A^e}{-}{A}$ to the standard bar
  resolution of $_AA_A$.

In \cite{SF08} Sánchez - Flores showed that the bracket can be defined
using the so-called \emph{reduced resolution} (or \emph{radical
  resolution}, see \cite{GS-86, C89}). We now follow \cite{SF08}. Let
$A = \Bbbk Q/I$ and $E$ be the semi-simple subalgebra of $A$ generated
by $Q_0$, the vertices if $Q$. As $E-E$-bimodules, we have that
$A = E \oplus r$ (where $r$ is the Jacobson radical of $A$). In the
remaining part of this section all tensor products are taken over
$E$. Let $r^{\otimes_{\K}^n}$ denote the $n^{\rm th}$ tensor power of $r$
with itself. One then has a projective resolution of $_A A_A$
\[
\mathbf{R}_\bullet \colon \cdots \to A \otimes_{\K} r^{\otimes_{\K} n} \otimes_{\K} A
\xrightarrow{\delta_{n-1}} \cdots \xrightarrow{\delta_1} A \otimes_{\K} r
\otimes_{\K} A \xrightarrow{\delta_0} A \otimes_{\K} A \xrightarrow{\mu} A \to 0\,,
\]
where $\mu$ is the multiplication of $A$ and

\[ \delta_{n-1} = \sum_{j=0}^n (-1)^j \mathbf{1}^{\otimes j} \otimes
\mu \otimes \mathbf{1}^{\otimes (n-j)}.\]

In order to compute the Hochschild cohomology spaces we apply the
functor $\Hom{A^e}{-}{A}$ to $\mathbf{R}_\bullet$, use the
identification

\begin{equation}
\label{eq:id-rad}
  \Hom{A^e}{A\otimes_{\K} r^{\otimes_{\K}^n} \otimes_{\K}}{A} \simeq
  \Hom{E^e}{r^{\otimes_{\K}^n}}{A}  
  \end{equation}

and denote the right-hand term by by $(r^n,A)$. The identification
allows to carry de differential, still denoted by $\delta_{\bullet}$,
giving a complex that we will denote by $(\mathbf{r}^\bullet,A)$.

Let $\pi : A \to r$ be the canonical projection of $A$ onto its
radical. Note that it is an $E-E$ morphism.

For $f\in (r^n,A),\ g\in (r^m,A)$ and $i\in\{1,\ldots, n\}$ let
$f\circ_ig\in (r^{n+m-1},A)$ be defined by the formula
\[ f\circ_i g = f \left( \mathbf{1}^{\otimes{i-1}}\otimes \pi  g \otimes \mathbf{1}^{n-i}\right).\]
Further, let
\[ f\circ g = \sum_{i=1}^n (-1)^{(i-1)(m-1)} f\circ_i g\]
and finally \[ [f,g] = f\circ g - (-1)^{(n-1)(m-1)} g\circ f.\]

This operation induces the so-called Gerstenhaber bracket in
$\HH{*}{A}$.

\subsection{The minimal resolution} The reduced resolution is smaller
than the standard one, but still too big to be used efficiently for
our purpose.

From now on, we will be interested in monomial quadratic algebras, so
we can use the minimal resolution of Bardzell \cite{B97} which has the
following description. Let $\AP{0}{} = Q_0,\ \AP{1}{} = Q_1$ and for
$n\geqslant 2,\ \AP{n}{} = \{ \a_1 \cdots \a_n| \a_i \in Q_1, \a_i
\a_{i+1} \in I \text{ for } 1\leqslant i < n\}$.
The elements of some $\AP{n}{}$ will be called $\AP{}{}-paths$. If
$x,y\in Q_0$ are fixed, we will write $\AP{n}{x,y}$ for the set of
$\AP{}{}$-paths going from $x$ to $y$. We know from \cite{GZ94} that
for every natural number $n$ and every fixed vertices $x,y\in Q_0$,
the set $\AP{n}{x,y}$ is in bijection with a basis for
$\Ext{A}{n}{S_x}{S_y}$.

For $n\geqslant 0$, denote by $\Bbbk \AP{n}{}$ the $\Bbbk$-vector
space with basis $\AP{n}{}$. With these notations we have a minimal
projective resolution of $_AA_A$.

\[
\mathbf{M}_\bullet \colon \cdots\to A\otimes_{\K} \K \AP{n}{} \otimes_{\K} A
\xrightarrow{\partial_{n-1}} \cdots \xrightarrow{\partial_1} A\otimes_{\K}
\K \AP{1}{}\otimes_{\K} A \xrightarrow{\partial_0} A \otimes_{\K} \K
\AP{0}{}\otimes_{\K} A \xrightarrow{\epsilon} A \to 0
\]
where $\epsilon$ is the composition of the isomorphism $A\otimes_{\K}\Bbbk \AP{0}{}\otimes_{\K}A \simeq A\otimes_{\K} A$ with the  multiplication of $A$. The differentials are given by
\[
\partial_{n-1}(1\otimes \a_1 \cdots \a_n \otimes 1) = \a_1 \otimes \a_2 \cdots \a_n \otimes 1 + (-1)^n \otimes \a_1 \cdots \a_{n-1} \otimes \a_n.
\]

We then apply $\Hom{A^e}{-}{A}$ to $\mathbf{M}_\bullet$ and make the
identification

\begin{equation}
  \label{eq:id-min}
  \Hom{A^e}{A\otimes_{\K} \Bbbk \AP{n}{} \otimes_{\K} A}{A} \simeq
\Hom{E^e}{\Bbbk \AP{n}{}}{A}.   
\end{equation}

We denote the right-hand term by $(\AP{n}{},A)$, and the resulting
complex by $(\mathbf{\Gamma}^\bullet,A)$.

In order to define the Gerstenhaber bracket using the complex
$(\mathbf{\Gamma}^\bullet,A)$, an explicit comparison of complexes may
be used. This approach was used in \cite[1.3 and 1.4]{Bus06-cohomo},
and in \cite{RR18} for quadratic string algebras but can be
generalized to all monomial quadratic algebras.

The comparison morphisms


\[ \begin{tikzcd} \mu_\bullet \colon A\otimes_{\K} \Bbbk \AP{n}{} \otimes_{\K} A =
  \mathbf{M}_\bullet \ar[r, shift left] &
  \mathbf{R}_\bullet = A\otimes_{\K} r^{\otimes_{\K}^n }\otimes_{\K} A\ar[l, shift
  left] \colon \omega_\bullet
\end{tikzcd}
\]

are defined as follows (recall that $E = \Bbbk Q_0 = \Bbbk \AP{0}{}$
and keep in mind that all the tensor products involved are taken over
$E$):

\begin{itemize}
\item[$\bullet$] The map $\mu_\bullet$ is defined by :
\[\begin{array}{rcl}
\mu_0(1\otimes e_i \otimes 1) & = & 1\otimes e_i = e_i\otimes 1.\\
\mu_n(1\otimes \a_1 \cdots \a_n\otimes 1) & = & 1 \otimes \a_1\otimes \a_2 \otimes \cdots \otimes \a_n \otimes 1 \text{ for } n\geqslant 1.\\
\end{array}\]
\item[$\bullet$] The map $\omega_\bullet$ is defined by
  \begin{itemize}
  \item
  $\omega_0(1\otimes 1) = 1 \otimes \sum_{i\in Q_0} e_o\otimes 1$,
\item
  $\omega_1(1\otimes \a_1\cdots \a_n \otimes 1) = \sum_{i=1}^n \a_1
  \cdots \a_{i-1} \otimes \a_i \otimes \a_{i+1}\otimes \cdots\otimes
  \a_n$
  for all paths $\a_1\cdots \a_n$ ($\a_i\in Q_1$),
\item  and
  $\omega_n(1\otimes p \otimes 1)$ is equal to $u\otimes v \otimes w$
  if the path $p$ decomposes as $uvw$ where $u,v,w$ are paths such
  that $v\in \AP{n}{}$, $u$ being of minimal length and $0$ otherwise.
  \end{itemize}
\end{itemize}

The maps $\mu_\bullet$ and $\omega_\bullet$ are both morphisms of
complexes and moreover $\omega_\bullet \mu_\bullet = \mathbf{1}$. Upon
applying $\Hom{A^e}{-}{A}$ and making the identifications
(\ref{eq:id-min}) and (\ref{eq:id-rad}) mentioned above, we obtain two
quasi-isomorphisms


\[ \begin{tikzcd}-\circ \mu^\bullet \colon (\mathbf{r}^\bullet,A) \ar[r, shift left] & (\mathbf{\Gamma}^\bullet,A) \colon - \circ \omega^\bullet \ar[l, shift left].
\end{tikzcd}
\]

This allows to define the operations in $(\AP{\bullet},A)$ which
will be used to investigate the Gerstenhaber bracket. More precisely,
given $f\in(\AP{n}{},A),\ g\in (\AP{m}{},A)$, and
$i\in \{1,\ldots, n\}$, one defines
\[f \circ_i g = \mu^{n+m-1}(\omega^nf \circ_i \omega^m g) \]
where the symbol $\circ_i$ on the right-hand term denotes the
operation defined using the reduced bar resolution. From this point,
define a bracket, still denoted $[-,-]$ on $(\mathbf{\Gamma}^\bullet, A)$ on $(\mathbf{\Gamma}^\bullet, A)$ following the same steps as for $(\mathbf{r}^\bullet, A)$ (see section \ref{sec:reduced-bar}).

The crucial point is that given $f\in (\AP{n}{},A), g\in (\AP{m}{},A)$
and $w = \a_1 \cdots \a_{n+m-1}\in \AP{n+n-1}{}$ then

\[ f\circ_i g (w) = f(\a_1 \cdots \a_{i-1} g(\a_i\cdots
\a_{i+m-1})\a_{i+m}\cdots \a_{n+m-1})
\]
whenever
$\a_1 \cdots \a_{i-1} g(\a_i\cdots \a_{i+m-1})\a_{i+m}\cdots
\a_{n+m-1}\in \AP{n}{}$,
and $f\circ_i g (w)$ vanishes otherwise.  Note that
$\alpha_i\cdots\alpha_{i+m-1}\in \Gamma^m$.

\section{Operations on paths, extensions and cycles}
\subsection{Compositions of paths}
Let $(Q,I)$ be a bound quiver with $I$ monomial and quadratic.

Let $m,n$ be positive integers. Let $r,s$ be integers such that
$1 \leqslant s \leqslant r \leqslant n$. We define an \emph{$(s,r)$
  $\Gamma$-bypass} to be a pair $(u,v)$ where
$u = \alpha_1 \cdots \alpha_n \in \AP{n}{}$ ($\alpha_i\in Q_1$),
$v = \beta_1 \cdots \beta_m\in \AP{m}{}$ ($\beta_i\in Q_1$), such that
$\alpha_1\cdots \alpha_{s-1} \beta_1 \cdots \beta_m \alpha_r \cdots
\alpha_n \in \AP{s+m+n-r}{}$.
For such a $(u,v)$ we then define the \emph{composition of $u$ and $v$
  at positions $s$ and $r$} as
\[u \csr{s}{r}   v = \a_1 \cdots \a_{s-1} \b_1 \cdots \b_m\a_r \cdots \a_n.\]

Informally, we think as ``\emph{start with $u$, then switch to $v$ for
  the $s^{\text{th}}$ arrow, go along $v$, then switch back to $u$
  when its $r^{\text{th}}$ arrow is hit}''.

We define a \emph{$\Gamma$-bypass} as a couple $(u,v)$ which is an
$(s,r)$ $\Gamma$-bypass for some $s,r$.

A particularly interesting case in view of the study of the
Gerstenhaber bracket is when $u\csr{s}{r} v \in \AP{n+m-1}{}$, that is
precisely when $r=s+1$. On this situation we shall write
$u \circ_{s} v$ instead of $u \csr{s}{s+1} v$.

Note also that since $u$ and $v$ are allowed to contain cycles, the
endpoints of $v$ can appear more than once while running along
$u$. Several compositions are thus possible in general.

\begin{Ex}
Let us consider the quiver

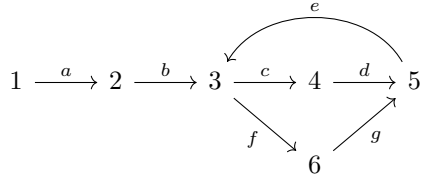
\[
\begin{tikzcd}
1 \ar[r,  "a"] & 2\ar[r, "b"] & 3 \ar[r, "c"] \ar[dr, "f"'] & 4 \ar[r, "d"] & 5\ar[ll,"e"', bend right=60] \\
                &               &                           & 6 \ar[ur, "g"'] &
\end{tikzcd}\]
bound by all the monomial relations of length 2. Let $u=abcdecd \in \AP{7}{}$ and $v = fg \in \AP{3}{}$. We then have $u \csr{3}{5} v = abfgecd \in \AP{7}{},\ u \csr{3}{8} v = abfg \in \AP{4}{}, u \csr{6}{8} v = abcdefg \in \AP{7}{} $.
\end{Ex}

\subsection{Paths and extensions} Let $x,y$ be vertices. It is well
known that an arrow of $Q$ (equivalently a path in $\AP{1}{x,y}$)
$\a \colon x \to y$ defines an extension between the simple modules
$S_x$ and $S_y$; this extension is as follows (the morphisms are the
natural ones)
\[
\E{\a}\colon 0 \to S_y \to M(\a) \to S_x \to 0\,.
\]
This construction on paths in $\AP{1}{x,y}$ extends to a construction
on paths in $\AP{n}{x,y}$, for all integers $n$. For all
$u = \a_1 \cdots \a_n \in \AP{n}{x,y}$ ($\a_i\in Q_1$) denote by
$\E{u}$ the following exact sequence where all the arrows are the
natural morphisms
  \[
  \E{u}\colon 0 \to S_y \to  M(\a_n) \to  M(\a_{n-1}) \to  \cdots
  \to M(\a_1)\to  S_x \to  0\,.
\]

Recall that the vector space $\mathrm{Ext}^n_A(S_x,S_y)$ has several
equivalent definitions. One of them defines this space as the set of
appropriate equivalence classes $\overline{\mathbf{E}}$ of $n$-fold extensions
$\mathbf{E}\colon 0 \to S_y \to M_n\to\cdots \to M_1\to S_x\to 0$. Another equivalent point of view is to consider it as the cohomology group $H^n\mathrm{Hom}_A(P,S_y)$, where $P$ is
any projective resolution of $S_x$ in $\mod{A}$.  In this text it will
be convenient to deal with the former description. The following lemma
is part of the folklore and follows from \cite{GZ94}. We provide a proof
for convenience. In what follows, given a path $p$ in $Q$ we denote by $\so(p)$ and $\tg(p)$ its source and its target.

\begin{Lem}
  The family $\{\overline{\E{u}}\}_{u\in \AP{n}{x,y}}$ is a basis of
  $\Ext{A}{n}{S_x}{S_y}$.
\end{Lem}
\begin{proof} Using a projective resolution $P$ of $S_x$ due to
\cite{GZ94}, we will make explicit the cohomology class in
$H^n\mathrm{Hom}_A(P,S_y)$ corresponding to a given
$u=\a_1\cdots\a_n\in \AP{n}{x,y}$ ($\a_i\in Q_1$). The conclusion
will follow from this description.  Recall that $S_x$ has of
projective resolution of the following form
  \[ \cdots \to P^{-\ell} \xrightarrow{d^{-\ell}} P^{-\ell+1} \to
\cdots \to P^{-1}\to P^0 \xrightarrow{\pi} S_x\,,
  \] where
  \begin{itemize}
  \item $\displaystyle P^{-\ell} = \bigoplus_{p \in \AP{\ell}{x,-}} e_{\mathfrak{t}(p)}A$ for all
$\ell \geqslant 0$,
  \item $\pi$ is the natural projection,
  \item for all $\ell \geqslant 1$, the differential $d^{-\ell}$ is
such that $d(e_{\tg(\beta_1\cdots\beta_\ell)}) = \beta_\ell$ ($\in
e_{\tg(\beta_1\cdots\beta_{\ell-1})}A$) for all $\beta_1\cdots
\beta_\ell\in \AP\ell{x,-}$ ($\beta_i\in Q_1$).
\end{itemize}
In terms of Bardzell's resolution, this is $S_x\otimes_A \mathbf{M}$.
Now, here is a natural morphism of complexes determined by $\E{u}$
\[
\SelectTips{eu}{10}\xymatrix{
  P^{-n} \ar[d]^{f^{-n}} \ar[r]
& P^{-n+1} \ar[r]
\ar[d]^{f^{-n+1}}
& \cdots \ar[r]
& P^{-\ell} \ar[r]
\ar[d]^{f^{-\ell}}
& P^{-\ell+1} \ar[r] \ar[d]^{f^{-\ell+1}}
& \cdots
\ar[r]
& P^{-1} \ar[r] \ar[d]^{f^{-1}}
& P^0 \ar[d]^{f^0} \\ S_y
\ar[r]
& M(\a_n) \ar[r]
& \cdots \ar[r]
& M(\a_{\ell+1}) \ar[r]
&
M(\a_{\ell}) \ar[r]
& \cdots \ar[r]
& M(\a_2) \ar[r]
& M(\a_1)\,,
}
\]
where
\begin{itemize}
\item $f^{-n} \colon P^{-n}\to S_y$ is the following composition of
  natural morphisms
  \[
  P^{-n} = \bigoplus_{p\in \AP{n}{x,-}} e_\tg(p) A \twoheadrightarrow
  e_{\tg(\a_1\cdots\a_n)} A \twoheadrightarrow S_{\tg(\a_n)}\,,
  \]
\item for all $\ell \in \{0,1,\ldots,n-1\}$, the morphism $f^{-\ell}
  \colon P^{-\ell} \to M(\a_{\ell+1})$ is the following composition of
  natural morphisms
  \[
  P^{-\ell} = \bigoplus_{p\in \AP{\ell}{x,-}} e_{\tg(p)} A
  \twoheadrightarrow e_{\tg(\a_1\cdots \a_{\ell})}A = e_{\so(\a_{\ell+1})}A
  \twoheadrightarrow M(\a_{\ell+1})\,.
  \]
\end{itemize}
Under the identification of equivalence classes of $n$-fold extensions
$0\to S_y\to \cdots \to S_x\to 0$ and cohomology classes in
$H^n \Hom{A}{P}{S_y}$, the equivalence class  $\overline{\E{u}}$
corresponds to the cohomology class of $f^{-n}$.

In view of this reminder,
$\{\overline{\E{u}}\}_{u\in \AP{n}{x,y}}$ is a basis of
$\Ext{A}{n}{S_x}{S_y}$ because $S_y$ is simple.
\end{proof}

In addition to $n$ and $x,y$, let $m$ be an integer and $x',y'$ be
vertices of $Q$.
We now introduce an operation, denoted by $\csr{s\ }{r}$,
\[ \csr{s\ }{r} \colon \Ext{A}{n}{S_x}{S_y} \otimes_\K
\Ext{A}{m}{S_{x'}}{S_{y'}} \to \Ext{A}{n-r+s-m}{S_x}{S_y}\,,\]
for all integers $s,r$ such that
$1 \leqslant s \leqslant r \leqslant n$. It is defined by its
behaviour on tensors of the shape $\E{u}\otimes_\K \E{v}$
($u\in \AP{n}{x',y'}$ and $v\in \AP{m}{x,y}$) as follows
\[
\E{u} \otimes_\K \E{v} \longmapsto
\begin{cases}
  \E{u \csr{s\, }{r} v} & \text{if $(u,v)$ is an $(s,r)$
    $\Gamma$-bypass,} \\
  0 & \text{otherwise.}
\end{cases}
\]
Here is an alternative description when $(u,v)$ is an $(s,r)$
$\Gamma$-bypass. We use the following notation
\begin{itemize}
\item $u = \a_1\cdots\a_n$ ($\a_i\in Q_1$), $v=\b_1\cdots\b_m$
  ($\b_i\in Q_1$),
\item $u' = \a_1\cdots \a_{s-1}$ ($\in \AP{s-1}{}$), $\widehat{u} =
  \a_s \cdots \a_{r-1}$ ($\in \AP{r-s}{}$), and $u'' = \a_r \cdots
  \a_n$ ($\in \AP{n-r+1}{}$).
\end{itemize}
Hence $u = u'\widehat u u''$ and $u \csr{s\ }{r}v = u'vu''$. In terms
of long exact sequences, $\E u$ is obtained by splicing $\E{u'}$,
$\E{\widehat u}$, and $\E{u''}$:

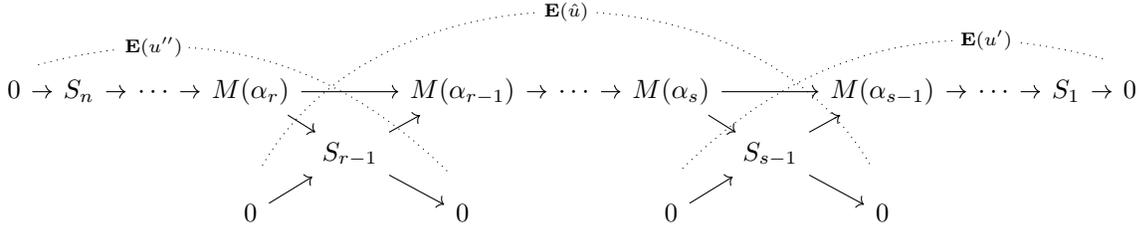
\begin{center}
\begin{tikzcd}[column sep=0.9ex, row sep = 1.6ex]
0 \ar[rr]\arrow[rrrrrrrrdd, no head, shift left = 2ex, bend left, near start,
"\E{u''}" description, dotted]	&& S_n \ar[rr] 	& &  \cdots \ar[rr] & &
M(\a_r)\ar[rr]\arrow[dr]	&& M(\a_{r-1})\ar[rr]   & & \cdots \ar[rr] && M(\a_{s})
\ar[rr]\ar[dr]  &\ & M(\a_{s-1})\ar[rr]&&  \cdots \ar[rr] & &  S_1
\ar[rr] & & 0 \\
 				& & 		& &  & & &
S_{r-1}\ar[ur]\ar[dr]& & & & & & S_{s-1} \ar[ur]\ar[dr]\\
&&  	& &   & &
0\ar[ur]\ar[rrrrrrrr, no head, bend left = 60, shift left = 2.5ex, dotted,
"\E{\hat{u}}" description, ]	&& 0  & & && 0\ar[ur]\ar[rrrrrrrruu, no head,
bend left, shift left = 2.5ex, dotted, "\E{u'}" description, near end]
  & &0&   & &  &&
 & &  \\
\end{tikzcd}
\end{center}
In terms of cup-products, this means that the following equalities
hold true
\[
\begin{array}{rl}
  \E{u} = \E{u''} \cup \E{\hat{u}} \cup \E{u'} & \text{ in }
                                                \Ext{A}{n}{S_x}{S_y}
  \\
  \E{u} \csr{s}{r} \E{v} : = \E{u''}\cup \E{v} \cup \E{u'}
  & \text{ in } \Ext{A}{n-r+s-m}{S_x}{S_y}.
\end{array}
\]

\subsection{Cochains and cycles}\label{sec:cochains-cycles}
In the previous sections we considered $\AP{}{}$-paths and
extensions. We now turn our attention to the Bardzell complex, through
which the Hochschild cohomology spaces are computed.

A non-zero cochain $f : \K \AP{n}{x,y} \to A$ sends and $\AP{}{}$-path
$u$ to a linear combination of paths from $x$ to $y$. We will see that
with some mild additional hypotheses, this context allows to obtain
what we call a \emph{cycle in $\mod{A}$}, that is, a diagram in
$\mod{A}$ of the shape $X_0\to X_1\to \cdots \to X_N$ where each
morphism is non-zero and non invertible, each $X_i$ is indecomposable, and $X_N\simeq X_0$.

Let $u =\a_1 \cdots \a_n \in\AP{n}{x,y}$ and assume there is a
non-zero path $p \in e_x A e_y$. We thus have a non-zero cochain, in
the Bardzell complex : it sends $u$ to $p$ and every other
$\AP{}{}$-path to zero. We denote it by $\chi^p_u$. Since the set of
these cochains is a $\K$-basis for $\Hom{E^e}{\K \AP{n}{x,y}}{A}$,
we call them \emph{basic cochains}, and their \emph{degree} is
$n$. A basic cochain $\chi^p_u$ of degree $n$ is said to be
\emph{reduced} if $p$ and $u$ do not start with the same arrow
$\a_1$ and do not end with the same arrow $\a_n$.

\begin{equation}
  \label{eq:2}
  \begin{tikzcd} x\ar[r,"\a_1"] \ar[rrrr, bend right = 30,"p"']&1
    \ar[r,"\a_2"]&\cdots \ar[r,"\a_{n-1}"] &n-1 \ar[r,"\a_n"]&y.
  \end{tikzcd}
\end{equation}

\medskip

  Notice that if $p$ is a stationary path, then $x=y$ and $u$ is an
  oriented cycle in $Q$. The oriented cycle $u$ then yields the following cycle in $\mod{A}$ which happens to be exact


   \[\begin{tikzcd}
 M(\a_n) \ar[r] & M(\a_{n-1}) \ar[r] & \cdots \ar[r] & M(\a_1) \ar[r]   & M(\a_n).
   \end{tikzcd}\]

Moreover, the kernels and cokernels of the morphisms appearing in this
cycle are precisely the simple modules corresponding to the sources
and targets of the arrows $\a_i$ for $i\in\{1,\ldots, n\}$ so that $\E{u}$ can be easily recovered from the given oriented cycle. The same holds for the extensions $\E{\a_i \cdots \a_n \a_1 \cdots \a_{i-1}}$  corresponding to the cycles obtained from $u$ by cyclic permutation.

Note also
that this holds regardless if $A$ is of finite representation type or
not. The cycle mentioned above is a cycle in $\mod{A}$, not
necessarily a cycle in the Auslander - Reiten quiver of $A$.

\medskip

In general it is still possible to associate a cycle in $\mod A$ to a
basic and reduced cochain. The construction is detailed below. The
resulting cycle has a specific shape and we give a name to it for
later purposes. We define an \emph{admissible cycle of degree $n$} in
$\mod A$ as a cycle of the following shape


\begin{equation}
  \label{eq:cyc-adm}
M(\a_{n-1}) \to M(\a_{n-2}) \to \cdots \to M(\a_2) \to
M(\a_1^{-1} p \a_n^{-1}) \to M(\a_{n-1})\,,
\end{equation}

where
\begin{itemize}
\item $\a_1\cdots\a_n$ is a path in $Q$ (with $\a_i\in Q_1$) with length $n$ at least $2$.
\item $p$ is a path parallel to $\a_1\cdots\a_n$ such that $\a_1$ is
  not a prefix of $p$ and $\a_n$ is not a suffix of $p$ and such that
  $p$ is non-zero in $A$, in particular the module
  $M(\a_1^{-1}p\a_n^{-1})$ is well-defined. This ensures that the arrows $\a_1,\ldots,\a_n$ and the path $p$ hence form a diagram the one shown in \eqref{eq:2}, above, from which we adopt the numbering of the vertices.

\item all the morphisms are the natural ones, note that
  \begin{itemize}
  \item $S_1$ is a direct summand of
    $\mathrm{soc}(M(\a_1^{-1} p \a_n^{-1}))$ and the map $M(\a_2) \to
    M(\a_1^{-1} p \a_n^{-1})$ is the composite morphism $M(\a_2)
    \twoheadrightarrow S_1 \hookrightarrow M(\a_1^{-1} p
    \a_n^{-1})$,
  \item $S_{n-1}$ is a direct summand of
    $\mathrm{top}(M(\a_1^{-1} p \a_n^{-1}))$ and the map $M(\a_1^{-1} p
    \a_n^{-1}) \to M(\a_{n-1})$ is the composite morphism $M(\a_1^{-1} p \a_n^{-1}) \twoheadrightarrow S_{n-1} \hookrightarrow M(\a_{n-1})$.
  \end{itemize}
\end{itemize}


Note that both the sequence of arrows $\a_1,\ldots,\a_n$ and the path
$p$ may be recovered from the admissible cycle. First, the sequence of two-dimensional modules $M(\a_{n-1}),M(\a_{n-2}),\ldots,M(\a_2)$ determine the arrows. Next,
$M(\a_1^{-1}p\a_n^{-1})$ is the only non uniserial module of the cycle  and it determines the modules $M(\a_1),\ M(\a_n)$ and $M(p)$ (hence the arrows $\a_1, \a_p$ and the path $p$):

\begin{itemize}
\item $M(\a_1)$ is the unique quotient of $M(\a_1^{-1}p\a_n^{-1})$ of
  the shape $M(p')$ for some path $p'$, and such that the socle and
  the top of $M(p')$ are direct summands of the socle and the top,
  respectively, of $M(\a_1^{-1}p\a_n^{-1})$,
\item $M(\a_n)$ is the unique submodule of
  $M(\a_1^{-1}p\a_n^{-1})$ of the shape $M(p')$ for some
  path $p'$, and such that the socle and the top of $M(p')$
  are direct summands of the socle and the top, respectively, of
  $M(\a_1^{-1}p\a_n^{-1})$,

\item finally $M(p)$ is the unique non zero cohomology group of
  (\ref{eq:cyc-adm}).
\end{itemize}
Thus, the modules $M(\a_1),\ldots,M(\a_n),M(p)$, and hence the of
arrows $\a_1,\ldots,\a_n$ and the path $p$ are determined by the
admissible cycle.
\bigskip

Now let $\chi^p_u$ be a basic and reduced cochain ($u=\a_1\cdots\a_n$,
$\a_i\in Q_1$) of degree at least two, and keep in mind the labelling of the vertices and arrows of \eqref{eq:2}. It follows from the definition
that the following diagram, where all the arrows are the natural
morphisms and which we will denote by $\Cy{u,p}$, is an admissible
cycle

%
%

\[\Cy{u,p} \colon  M(\a_{n-1}) \to M(\a_{n-2}) \to \cdots \to M(\a_2) \to
M(\a_1^{-1} p \a_n^{-1}) \to M(\a_{n-1})\,.
\]

Note that this sequence is exact except at
$M(\a_1^{-1}p\a_{n-1}^{-1})$. Also, the kernel of the map with domain
$M(\a_j)$ is the simple module $S_j$ for $j\in\{2,\ldots, n-1\}$ and
$S_1$ is the image of $M(\a_2)\to M(\a_1^{-1}p\a_n^{-1})$. It follows
from this remark that the admissible cycle $\Cy{u,p}$ alows to obtain cycles starting and ending at each of theses simple modules.

On the other hand, in case $n=2$ we obtain an admissible
cycle that is reduced to one single module $M(\a_1^{-1} p \a_2^{-1})$
with a non invertible non trivial endomorphism (see the example
below).

The cycle $\Cy{u,p}$ may be recovered using homological algebra as
follows. Denote $\a_1\cdots\a_{n-1}$ by $u'$, this is the only suffix
of $u$ lying in $\AP{n-1}{}$; its associated $(n-1)$-fold extension is
\[
\E{u'} \colon 0 \to S_{n-1} \to M(\a_{n-1}) \to M(\a_{n-2})\to
\cdots \to M(\a_1)\to S_x\to 0\,.
\]

The module $M(p\a_n^{-1})$ is well-defined because $p$ does not end
with $\a_n$.  Moreover, its top contains $S_x$ as a direct summand because $p$ starts at $x$. Denote by $\pi_x$ the natural surjection
$M(p\a_n^{-1}) \twoheadrightarrow S_x$. Considering the pullback of
$\E{u'}$ along $\pi_x$ yields the following commutative diagram whose
rows are exact


\[\begin{tikzcd}[column sep = small]
 		& 0  \ar[r] 	&S_{n-1} \ar[r]\ar[d, equal] &M(\a_{n-1}) \ar[r]\ar[d, equal] & \cdots \ar[r]& M(\a_2)\ar[d, equal] \ar[r] & M(\a_1^{-1} p \a_n^{-1}) \ar[r]\ar[d] & M(p \a_n^{-1}) \ar[r]\ar[d,"\pi_x"] & 0   \\
\E{u'}: &      0 \ar[r] &S_{n-1} \ar[r] & M(\a_{n-1}) \ar[r]& \cdots \ar[r]& M(\a_2) \ar[r] & M(\a_1)\ar[r] & S_x \ar[r] & 0.
\end{tikzcd} \]

Note that $S_{n-1}$ is also a direct summand of
$\mathrm{top}\ M(p\a_n^{-1})$. Now, keeping the first row in the above pullback
diagram and composing the rightmost non-zero morphism with the natural surjection $M(p\a_n^{-1}) \twoheadrightarrow S_{n-1}$ and the natural inclusion $S_{n-1}\hookrightarrow M(\a_{n-1})$ yields the cycle $\Cy{u,p}$.

\medskip

In our homological construction we started by considering $u' = \a_1 \cdots \a_{n-1}$, but could also start with $u'' = \a_2 \cdots \a_n$, the unique suffix of $u$ belonging to $\AP{n-1}{}$, starting at the vertex $1$, the target of $\a_1$,  then consider its associated extension $\E{u''}$ and the pushout with the map $S_{1} \to M(\a_1^{-1}p)$. This leads, again, to consider the module $M(\a_1^{-1}p \a_n^{-1})$. The cycle is then completed with the map $S_y \to M(\a_1^{-1} p)$, and we are led to

\[\begin{tikzcd}[column sep = small]
 		M(\a_2) \ar[r] & M(\a_1^{-1}p \a_n^{-1})\ar[r] &M(\a_{n-1}) \ar[r]& \cdots \ar[r]& M(\a_2).
\end{tikzcd}
\]

The cycle obtained in this way is \emph{``rotated''}, but it uses the same maps and modules. 


\begin{Ex} \
Let $A = \K Q / I$ where $Q$ is the quiver
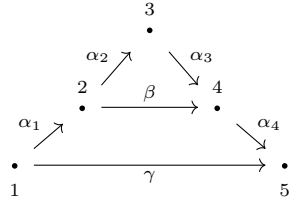
\begin{center}
\begin{tikzcd}[column sep=4mm, row sep = 2mm]
 & 		&\Vtx{3}\ar[dr,"\a_3"] & &\\
 & \Vtx{2} \ar[ur, "\a_2"] \ar[rr, "\b"]	& 			&\Vtx{4}
\ar[dr, "\a_4"] & \\
\vtx{1} \ar[ur, "\a_1"] \ar[rrrr, "\gamma" below] & & & &\vtx{5}
 \end{tikzcd}
\end{center}
and $I$ is the ideal generated by the paths of length 2. The algebra $A$ is then a representation - finite string algebra, and one can compute

\[ \dim{\HH{i}{A}} = \begin{cases}
  1 & \text{ if } i\in \{0,2,3,4\}\\
  2 & \text{ if } i = 1,\\
  0 & \text{ otherwise.}
\end{cases}\]

The generators of the Hochschild spaces are given by (the classes of)
the basic cochains $\chi^\b_\b$ and $\chi^\c_\c$ for $i=1$;
$f^2 = \chi^{\b}_{\a_2 \a_3}$ for $i=2$,
$f^3 = \chi^{\c}_{\a_1 \b \a_4}$ for $i=3$ and
$f^4 = \chi^{\c}_{\a_2\a_2 \a_3\a_4}$ for $i=4$. 

The Auslander-Reiten quiver of $A$ is depicted below. We indicate the
modules by their composition factors and the dotted lines are the
Asulander-Reiten translation.

 
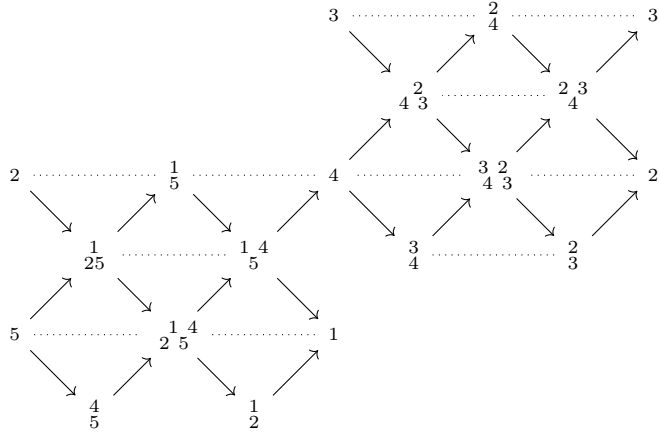
\begin{center}
 \begin{tikzpicture}[node distance = 1.5cm,<-]
 \node 				(s2) {$\sm{2}$};

 \node[below right of = s2] (p1) {$\sm{\ 1 \ \\ 25}$} edge (s2);
 \node[below left of = p1] (s5) {$\sm{5}$} edge[->] (p1); \node[below
 right of = s5] (p4) {$\sm{4\\5}$} edge (s5);

 \node[above right of = p1] (p15) {$\sm{1\\5}$} edge (p1) edge[-,
 dotted] (s2); \node[below right of = p1] (M) {$\sm{\ \ 1\ 4\\2\ 5}$}
 edge (p1) edge (p4) edge[-,dotted] (s5);

 \node[below right of = p15] (i5) {$\sm{1\ 4\\ 5}$} edge(p15) edge (M)
 edge[-,dotted] (p1); \node[below right of = M] (i2) {$\sm{1\\2}$} edge
 (M);

 \node[above right of = i5] (s4) {$\sm{4}$} edge (i5) edge[-, dotted]
 (p15); \node[above right of = s4] (p2) {$\sm{\ 2\\ 4\ 3}$} edge (s4);
 \node[above left of = p2] (s3) {$\sm{3}$} edge[->] (p2); \node[above
 right of = p2] (ts3) {$\sm{2\\ 4}$} edge(p2) edge[-,dotted] (s3) ;

 \node[below right of = ts3] (tp2) {$\sm{2\ 3\\ 4}$} edge
 (ts3)edge[-,dotted] (p2) ; \node[above right of= tp2] (Cs3) {$\sm{3}$}
 edge(tp2) edge[-,dotted] (ts3) ;

 \node[below right of = p2] (ts4) {$\sm{3\ 2\\\ 4\ 3}$} edge(p2)
 edge[->] (tp2) edge[-,dotted] (s4); \node[below right of = s4] (m34)
 {$\sm{3\\4}$} edge(s4) edge[->](ts4);

 \node[below right of = i5] (s1) {$\sm{1}$} edge (i5) edge (i2) edge[-,
 dotted] (M);

 \node[below right of = ts4] (i3) {$\sm{2\\3}$} edge (ts4)
 edge[-,dotted] (m34) ; \node[above right of = i3] (Cs2) {$\sm{2}$}
 edge(i3) edge(tp2) edge[-,dotted] (ts4);



 -1mm]Cs3.south);

 \end{tikzpicture}
 \end{center}

Further, the admissible cycles corresponding to these generators are given below.  Note that the cycle corresponding to $f^2 = \chi^\b_{\a_2\a_3}$ is reduced to a single module. We indicate, in each case, the admissible cycle associated with each cochain, as well as a corresponding cycle starting and ending at a simple module.


\begin{equation}
\label{eq:cyclesAR}
\begin{tikzcd}[row sep = 0.1ex]
\Cy{\b, \a_2\a_3} : & & \sm{3\ 2\\\ 4\ 3}&  &\\
\text{Cycle at }S_3 :  & 	S_3 \ar[r] &\sm{3\ 2\\\ 4\ 3} \ar[r] &S_3 & \\
 \ 					& & & &\\
\Cy{a_1 \b \a_4, \c} :&  &\sm{2\\4}\ar[r] & \sm{1\ 4\\\ 2\ 5} \ar[r] &\sm{2\\4} &\\
\text{Cycle at }S_4 :& S_4 \ar[r]&\sm{2\\4}\ar[r] & \sm{1\ 4\\\ 2\ 5} \ar[r] & S_4\\
 \ 					& & & &\\
\Cy{\a_1\a_2\a_3\a_4,\c} :& &\sm{3\\4}\ar[r] & \sm{2\\3}\ar[r] &\sm{1\ 4\\\ 2\ 5} \ar[r] &\sm{3\\4}\\
\text{Cycle at }S_4 : &S_4 \ar[r] &\sm{3\\4}\ar[r] & \sm{2\\3}\ar[r] &\sm{1\ 4\\\ 2\ 5} \ar[r] &S_4.\\
\end{tikzcd}
\end{equation}
\end{Ex}

\medskip

Clearly this construction does not apply to non-reduced basic
cochains. When $\chi_u^p$ is not reduced one may try to associate a
cycle to it by performing the above construction to $\chi_v^q$ where
$(v,q)$ is obtained from $(u,p)$ by removing from $u$ and $p$ their
largest common suffix and prefix. However, this is not always
possible because the resulting reduced basic cochain $\chi_v^q$ may
have degree smaller than $2$.
\begin{Ex}
Let $Q$ be the quiver
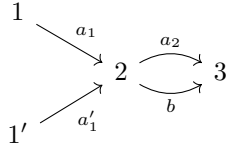
\[
\begin{tikzcd}[row sep = small]
 1 \ar[dr, "a_1"] & & \\
 		 & 2 \ar[r, bend left, "a_2"] \ar[r, bend right, "b"']& 3\\
 1' \ar[ur, "a'_1"'] & &
\end{tikzcd}
\]
bound by $I= \langle a_1 a_2, a'_1 a_2 \rangle$. The cochain
$\chi^{a_1 b}_{a_1a_2}$ is not reduced. Removing from $a_1b$ and
$a_1a_2$ their largest common prefix and suffix yields the reduced
cochain $\chi^{b}_{a_2}$. The latter has degree $1$. Hence, the
previous construction of admissible cycles does not apply to it.
\end{Ex}

We now turn our attention to the compositions of $\AP{}{}$-paths, and
how these are reflected by the cycle assignment, and to the
compositions of cochains that give rise to the Gerstenhaber bracket.

\subsection{Compositions of cochains}

Let $\chi^p_u$ and $\chi^q_v$ be two basic cochains with
$u=\a_1\cdots \a_n \in \AP{n}{}$ and $v=\b_1\cdots \b_m\in \AP{m}{}\ (\a_i, \b_j \in Q_1)$.
Saying that the cochain $\chi^p_u \circ_s \chi^q_v$ is non-zero
amounts to say that there is a path
$\c_1 \cdots \c_{n+m-1}\in \AP{n+m-1}{}$ such that
\[\chi^p_u \circ_s \chi^q_v (\c_1 \cdots \c_{n+m-1}) \ne 0.\]

Since the only non-zero value that $\chi^p_u$ takes on $\AP{}{}$-paths
is $p$, this means that
\[ \chi^p_u(\c_1 \cdots \c_s \chi^q_v(\c_s\cdots
\c_{s+m-1})\c_{s+m}\cdots \c_{n+m-1}) = p\]
and this forces $\c_1 \cdots \c_s q\c_{s+m}\cdots \c_{n+m-1} = u$ so
$(u,v)$ is a $\AP{}{}$-bypass, and further
\[\chi^p_u \circ_s \chi^q_v = \chi^p_{u\circ_s v}.\]

We thus have a situation like the one illustrated in \eqref{eq:1}.

\medskip


\begin{equation}
  \label{eq:1}
 \begin{tikzcd}
	      & & & \ar[r, bend left = 60, dotted, no head]           & \ar[d, "\b_{m-1}"]    & &\\
	      & & & b\ar[u,"\b_2"]           &c \ar[d, "\b_m"]    & &\\
  x \ar[r, "\a_1"] \ar[rrrrrr, bend right=25, "p"']&1\ar[r] & \cdots \ar[r, "\a_{s-1}"]& s-1 \ar[r, "\a_s = q"] \ar[u,"\b_1"]&s  \ar[r, "\a_{s+1}"]& \cdots \ar[r, "\a_n"] & y
   \end{tikzcd}
\end{equation}

\medskip

It follows from the linearity in the definition of the bracket at the cochain level that:

\begin{Prop} Let $\chi_u^p$ and $\chi_v^q$ be basic cochains of  degrees $n$ and $m$, respectively. Let $s$ be an integer. Then
  \[
  \chi_u^p {\csr{\ }{s}} \chi_v^q =
  \begin{cases}
    \chi_{u{\csr{\ }{s}}v}^p & \text{if $(u,v)$ is an $(s,s+1)$
      $\Gamma$-bypass}\\
    0 & \text{otherwise.}
  \end{cases}
  \]
\end{Prop}\qed

\subsection{Compositions of cycles}

We now turn our attention to cycles in $\mod{A}$. We will define 
operations between admissible cycles called \emph{compositions} (at some specific place) and we
will see that for reduced basic cochains $\chi_u^p$ and $\chi_v^q$ of
degree at least two and such that $(u,v)$ is an $(s,s+1)$
$\Gamma$-bypass, the composition of the admissible cycles $\Cy{u,p}$
and $\Cy{v,q}$ is the cycle $\Cy{u\circ_s v,p}$.

\medskip

Consider two admissible cycles
\[
\begin{array}{ll}
\mathbf{C}\colon &  M(\a_{n-1}) \to \cdots \to M(\a_2) \to
M(\a_1^{-1} p \a_n^{-1}) \to M(\a_{n-1}) \\
\mathbf{C}'\colon &   M(\b_{m-1}) \to \cdots \to M(\b_2) \to M(\b_1^{-1} q \b_m^{-1}) \to M(\b_{m-1})
\end{array}
\]
and assume there is an integer $s\in \{1,\ldots, n\}$ satisfying :
\begin{enumerate}[(a)]
\item $\b_1\cdots\b_m$ is parallel to $\a_s$
\item $q=\a_s$.
\end{enumerate}
Recall from subsection~\ref{sec:cochains-cycles} that the admissible
cycles $\mathbf{C}$ and $\mathbf{C}'$ together with $s$ determine
$\b_1\cdots\b_m$, $\a_s$ and $q$, hence (a) and (b) are indeed assumptions
on the admissible cycles $\mathbf C$ and $\mathbf{C}'$. We hence have
the diagram \eqref{eq:1}

The definition of the \emph{$s$-th composition} of $\mathbf{C}$ with
$\mathbf{C}'$, denoted by $\mathbf{C}\circ_s \mathbf{C}'$, is given
below according to whether $s=1$, $1<s<n$, or $s=n$.  Recall from
subsection~\ref{sec:cochains-cycles} that the cycles $\mathbf{C}$ and
$\mathbf{C}'$ determine the sequences of modules
$M(\a_1),\ldots,M(\a_n),M(p)$ and $M(\b_1),\ldots,M(\b_m),M(q)$. In
particular, $\mathbf{C}\circ_s\mathbf{C}'$ depends on $\mathbf{C}$ and
$\mathbf{C}'$ only.

If $s=1$ then $\mathbf{C}\circ_s \mathbf{C}'$ is defined as the following
admissible cycle where all the arrows are the natural morphisms


\[
\begin{tikzcd}[column sep = small]
 M(\a_{n-1})\ar[r]
& M(\a_{n-2}) \ar[r]
& \cdots \ar[r]
& M(\a_2) \ar[d]
&
& M(\b_1^{-1} p \a_n^{-1}) \ar[r]
& M(\a_{n-1})\\
&
&
& M(\b_m) \ar[ld]
&
& \\
&
& M(\b_{m-1}) \ar[r]
& \cdots
&
& \cdots \ar[r]
& M(\b_2) \ar[uul]
\end{tikzcd}
\]

This admissible cycle involves the module $M(\b_1^{-1}p\a_n^{-1})$. As an
object of $\mod A$, this module is obtained from $\mathbf{C}$ and
$\mathbf{C}'$ in two steps as follows
\begin{enumerate}[1.]
\item $M(p\a_n^{-1})$ is the cokernel of the natural morphism
  $\mathrm{rad}(M(\a_1)) \to M(\a_1^{-1}p\a_n^{-1})$,
\item $M(\b_1^{-1}p\a_n^{-1})$ is the middle term of a short exact
  sequence of the shape
	
  \[0 \to \mathrm{rad}(M(\b_1)) \to M(\b_1^{-1}p\a_n^{-1}) \to
  M(p\a_n^{-1})\to 0.\]
\end{enumerate}

If $1<s<n$ then $\mathbf{C}\circ_s \mathbf{C}'$ is defined as the following
admissible cycle where all the arrows are the natural morphisms

\[
 \begin{tikzcd}[column sep = tiny]
M(\a_{n-1})\ar[r] & \cdots \ar[r]& M(\a_{s+1})\ar[d]  & M(\a_{s-1})\ar[r]& \cdots \ar[r] & M(\a_1^{-1} p \a_n^{-1}) \ar[r]& M(\a_{n-1})\\
				&				& M(\b_{m})\ar[dl]  & M(\b_1)\ar[u]& 	& &\\
				&M(\b_{m-1})\ar[r]&\cdots  &\cdots\ar[r] &M(\b_2)\ar[ul]
\end{tikzcd}
\]
If $s=n$ then $\mathbf{C}\circ_s \mathbf{C}'$ is the following
admissible cycle where all the arrows are the natural morphisms

\[\begin{tikzcd}[column sep = small]
  & M(\a_1^{-1}p\b_m^{-1}) \ar[ldd]
  &
  & M(\a_{n-1}) \ar[r]
  & \cdots \ar[r]
  & M(\a_2) \ar[r]
  & M(\a_1^{-1}p\b_m^{-1})
  & \\
  &
  &
  & M(\b_1) \ar[u] \\
  M(\b_{m-1}) \ar[r]
  & \cdots
  &
  & \cdots \ar[r]
  & M(\b_2) \ar[lu]
 \end{tikzcd}
\]

This admissible cycle involves the module $M(\a_1^{-1}p\b_n^{-1})$. As an
object of $\mod A$, this module is obtained from $\mathbf{C}$ and
$\mathbf{C}'$ in two steps as follows
\begin{enumerate}[1.]
\item $M(\a_1^{-1}p)$ is the kernel of the natural surjection
  $M(\a_1^{-1}p\a_n^{-1})\twoheadrightarrow
  M(\a_n)/\mathrm{soc}(M(\a_n))$,
\item $M(\a_1^{-1}p\b_n^{-1})$ is the middle term of a short exact
  sequence of the shape
  \[0 \to M(\a_1^{-1}p) \to M(\a_1^{-1}p\b_n^{-1}) \to
  M(\b_n^{-1})/\mathrm{soc}(M(\b_n^{-1}))\to 0.\]
\end{enumerate}

We are now able to state our main result.

\begin{Thm} Let $\chi^p_u$ and $\chi^q_v$ be two reduced basic cochains
  of degrees $n\geqslant 2$ and $m\geqslant 2$ respectively. Let $s$
  be an integer such that $1\leqslant s \leqslant n$. Assume that
  $(u,v)$ is an $(s,s+1)$ $\Gamma$-bypass. Then, the $s$-th
  composition of $\Cy{u,p}$ with $\Cy{v,q}$ is defined and
  $\Cy{u,p}\circ_s \Cy{v,q} = \Cy{u\circ_s v,p}$.
\end{Thm}
\begin{proof}
This follows readily from the definition of the compositions $u\circ_s
v$ and $\Cy{u,p}\circ_s\Cy{v,p}$. Note that, by assumption, $q$ is the
$s$-th arrow of the path $u$.
\end{proof}

\begin{Ex} Consider the first example of section \ref{sec:cochains-cycles}. The generators of the Hochchild comology groups were already described. Further, a direct computation shows that $f^3 \circ_2 f^2 = f^4$, and that $f^3 \circ f^2 = - f^3 \circ_2 f^2$, so that $[\HH{3}{A}, \HH{2}{A}] = \HH{4}{A}$.

From the module categoy point of view, one sees that there are two cycles at the simple module $S_4$. One of them corresponds to $\Cy{\a_1 \b \a_4,\c}$, the other to $\Cy{\a_1 \a_2 \a_3  \a_4,\c}$. The second is obtained by composing the first one with the cycle at $S_3$. See the diagrams (\ref{eq:cyclesAR}).

\end{Ex}

\begin{Rmk} While our original interest was to understand the
  Gerstenhaber bracket in Hochschild cohomoloy, it must be noted that
  a non-zero element in some $\HH{j}{A}$ for $j\geqslant 2$ is not
  necessary to obtain oriented cycles in $\mod{A}$. Indeed, our
  construction requires only a reduced basic cochain $\chi^p_u$, it
  may not be a cocycle, or it could be a coboundary. Note that in this
  situation, the subcategory formed by the arrows of $u$ and those of
  $p$ form a \emph{clockwork cycle}, in the sense of \cite{P85}, where
  the representation-finite case was studied. Therein it is shown that
  a clockwork cycle gives rise to a cycle in the Auslander - Reiten
  quiver of $A$, so cycles do appear in a broader setting than that of
  Hochschild cohomology.  
\end{Rmk}

\section*{Acknowledgements}
This work presented in this article was partially developed during visits of the second named author to the first named author and to Ibrahim Assem at the Mathematics Department of Universit\'e de Sherbrooke. He acknowledges Ibrahim Assem's and Juan Carlos Bustamante's warm hospitality.

The first named author is grateful to the Mathematics Department of Université de Sherbrooke, where a considerable amount of this work was done.

\bibliographystyle{acm}
\bibliography{biblio}

\end{document}